\documentclass[12pt]{amsart}
\makeatletter
\newcommand*{\rom}[1]{\expandafter\@slowromancap\romannumeral #1@}
\makeatother
\usepackage{etex}
\usepackage{amsmath,amsthm,amsfonts,amssymb,tikz-cd,mathdesign,float}
\usepackage{lmodern}
\usepackage{relsize}
\usepackage{euler}
\usepackage{times}
\usetikzlibrary{positioning,automata}
\usepackage{mathtools}
\usepackage{youngtab}
\usepackage{pgfplots}
\usepackage{pst-plot}
\usepackage{pgfplots}
\usepackage[alphabetic]{amsrefs}
\pgfplotsset{compat=1.10}
\usepgfplotslibrary{fillbetween}
\usetikzlibrary{patterns}
\usetikzlibrary{matrix}
\numberwithin{equation}{section}

\newtheorem{thm}{Theorem}[section]
\newtheorem{prop}[thm]{Proposition}

\newtheorem{cor}[thm]{Corollary}
\usepackage{xcolor}
\theoremstyle{definition}

\newtheorem{rmk}[thm]{Remark}

\usepackage[pdftex,colorlinks=true,allcolors=blue]{hyperref}
\address{Department of Mathematics, Northeastern University,
	Boston, MA, 02115, USA}
\email{tsvelikhovskiy.b@husky.neu.edu}
\DeclareMathAlphabet\mathbfcal{OMS}{cmsy}{b}{n} 
\begin{document}

\title{On Poincare Polynomials  of  Shuffle Algebra Representations}
\author{Boris Tsvelikhovsky}
\maketitle
\begin{abstract}
In this paper we study some combinatorial properties of the shuffle algebra $S_{q_{1},q_{2}}$ with two complex parameters and, in particular, obtain the Hilbert series for $S_{q_{1},q_{2}}/I$, where $I$ is the ideal generated by $z-\lambda$ for both generic $q_{1},q_{2}$ and satisfying the relation $q_{1}^{a}q_{2}^{b}=1$ with $a, b \in \mathbb{N}$. 
\end{abstract}
{\hypersetup{linkcolor=black}
\tableofcontents
}
\section{Introduction}
\par
The shuffle algebras were introduced by Feigin and Odesskii. These algebras are unital associative subalgebras of $\mathbb{C}\underset{n\in\mathbb{Z}_{>0}}{\bigoplus}\mathbb{C}(z_{1},\hdots,z_{n})^{\mathfrak{S}_{n}}$ with multiplication defined by $$f(z_{1},\hdots,z_{n})*g(z_{1},\hdots,z_{m}):=\mbox{Sym}\left(f(z_{1},\hdots,z_{n})*g(z_{n+1},\hdots,z_{n+m})\prod\limits_{\substack{i\in \{1,\hdots n\} \\ j\in \{n+1,\hdots n+m\}}}\mu\left(\frac{z_{i}}{z_{j}}\right)\right),$$ for some function $\mu$. In the work of Schiffmann and Vasserot  \cite{Sch-V} it was shown that subalgebras in Hall algebras of vector bundles of smooth projective curves generated by $\mathbf{1}_{Pic^{d}(X)}:=\sum\limits_{\mathcal{L}\in Pic^{d}(X)}\mathcal{L}$, where $Pic^{d}(X)$ is the set of line bundles over $X$ of degree $d$ are isomorphic to subalgebras of $S$, generated by elements of degree $1$. For a smooth projective curve of genus $g$, one takes $\mu_{g}(x)=x^{g-1}\frac{1-qx}{1-x^{-1}}\prod\limits_{i=1}^{g}(1-\alpha_{i}x^{-1})(1-\overline{\alpha_{i}}x^{-1})$, where $\alpha_{i}$ are the roots of the numerator of the zeta function of the curve, i.e.  $\zeta_{X}(t)=\mbox{exp}(\sum\limits_{d \geq 1}\#X(\mathbb{F}_{q^{d}})\frac{t^{d}}{d})=\frac{\prod\limits_{i=1}^{g}(x-\alpha_{i})(x-\overline{\alpha_{i}})}{(1-t)(1-qt)}$. This isomorphism was made more explicit in case of elliptic curves (elliptic Hall algebras) in \cite{Neg}. In particular, the corresponding shuffle algebra was shown to be generated by polynomials of degree $1$. 
\par
The action of shuffle algebra on the
sum of localized equivariant $K$-groups (with respect to the $T=\mathbb{C}^{*}\times \mathbb{C}^{*}$ induced from the action on $\mathbb{C}^{2}$) of Hilbert schemes of points on $\mathbb{C}^{2}$ was provided in  \cite{F-Ts}.

\par
In present paper we study the Hilbert series of $S_{q_{1},q_{2}}/I$ and $\widetilde{S}_{q_{1},q_{2}}/I$, where $S_{q_{1},q_{2}}$ is a subalgebra of symmetric rational functions over the field $\mathbb{C}(q_{1}, q_{2})$ and $I$ - an ideal, generated by $(z-\lambda)$ for some nonzero $\lambda \in \mathbb{C}$. The precise definitions will be given in subsequent sections. It is shown that these Hilbert series have nice combinatorial counterparts. 
\par
The second section gives the definition and basic properties of the algebras in study.
\par
Section $3$ contains the results for generic and nongeneric $q_{1}$, $q_{2}$, i.e. $q_{1}^{a}q_{2}^{b}=1$ for $S_{q_{1},q_{2}}/I$.  
\par
Section $4$ is devoted to the proof that the Hilbert series for a certain representation of $S^{+}_{q_{1},q_{2}}/I$, the 'positive' subalgebra of $S_{q_{1},q_{2}}$, coincides with the generating function for Catalan numbers.
\par
\textbf{Acknowledgements.}
I would like to thank Boris Feigin for introducing me to the subject, plenty enlightening discussions and numerous helpful suggestions.
\section{Shuffle Algebra $S_{q_{1},q_{2}}$}
\par
We consider the shuffle algebra depending on two parameters $q_{1}, q_{2}$. This
is an associative graded unital algebra $S_{q_{1},q_{2}}:=\bigoplus\limits_{n\geq 0}S_{n}$ over the field $\mathbb{C}(q_{1},q_{2})$, where each component $S_{n}$ consists
of rational functions $F(z_{1}, \hdots, z_{n}) = \frac{f(z_{1}, \hdots, z_{n})}{\triangle_{n}}$, with $f(z_{1}, \hdots, z_{n})$ - a symmetric Laurent polynomial and $\triangle_{n}=\prod\limits_{i<j}(z_{i}-z_{j})$ - the Vandermonde determinant. The product of $F(z_{1}, \hdots, z_{n}) = \frac{f(z_{1}, \hdots, z_{n})}{\triangle_{n}}$ and $G(z_{1}, \hdots, z_{m}) = \frac{g(z_{1}, \hdots, z_{m})}{\triangle_{m}}$ is given by $$F(z_{1}, \hdots, z_{n})*G(z_{1}, \hdots, z_{m})=\mbox{Alt}_{S_{n+m}}\left(F(z_{1}, \hdots, z_{n})G(z_{n+1}, \hdots, z_{n+m})\prod\limits_{\substack{i\in \{1,\hdots n\} \\ j\in \{n+1,\hdots n+m\}}}\mu(z_{i},z_{j})\right),$$ where $\mu(x,y)=\frac{(x-q_{1}y)(x-q_{2}y)}{(x-y)^{2}}$ and $\mbox{Alt}_{S_{k}}\left(H(z_{1},\hdots,z_{k})\right)=\frac{1}{k!}\sum\limits_{\sigma \in S_{k}}sgn(\sigma)H(z_{\sigma(1)},\hdots, z_{\sigma(k)})$.
\par
\begin{prop}
\label{Prop1} 
The algebra $S_{q_{1},q_{2}}$ with the $*$-product is associative.  
\end{prop}
\begin{proof}
Let $F\in S_{l}$, $G\in S_{m}$ and $H\in S_{n}$. Then $$(l+m+n)!(F*G)*H=\sum\limits_{\sigma \in S_{l+m+n}}sgn(\sigma)\sigma\left((F*G)H\prod\limits_{\substack{i\in \{1,\hdots, l+m\} \\ j\in \{l+m+1,\hdots, l+m+n\}}}\mu(z_{i},z_{j})\right)=$$ $$=\sum\limits_{\sigma \in S_{l+m+n}}sgn(\sigma)\sigma\left(\sum\limits_{\tau \in S_{l+m}}sgn(\tau)\left(\frac{1}{(l+m)!}FG\prod\limits_{\substack{i'\in \{1,\hdots, l\} \\ j'\in \{l+1,\hdots, l+m\}}}\mu(z_{i'},z_{j'})\right)H\prod\limits_{\substack{i\in \{1,\hdots, l+m\} \\ j\in \{l+m+1,\hdots, l+m+n\}}}\mu(z_{i},z_{j})\right)=$$ $$=\sum\limits_{\sigma \in S_{l+m+n}}sgn(\sigma)\sigma\left(FGH\prod\limits_{\substack{i'\in \{1,\hdots, l\} \\ j'\in \{l+1,\hdots, l+m\}}}\mu(z_{i'},z_{j'})\prod\limits_{\substack{i\in \{1,\hdots, l+m\} \\ j\in \{l+m+1,\hdots, l+m+n\}}}\mu(z_{i},z_{j})\right).$$
\par
Performing analogous manipulations, one can show that $$(l+m+n)!F*(G*H)=\sum\limits_{\sigma \in S_{l+m+n}}sgn(\sigma)\sigma\left(FGH\prod\limits_{\substack{i'\in \{1,\hdots, l\} \\ j'\in \{l+1,\hdots, l+m+n\}}}\mu(z_{i'},z_{j'})\prod\limits_{\substack{i\in \{l+1,\hdots, l+m\} \\ j\in \{l+m+1,\hdots, l+m+n\}}}\mu(z_{i},z_{j})\right).$$
\par
The above implies the equality $(F*G)*H=F*(G*H)$.
\end{proof}
\par
We will need a few more properties of the algebra $S_{q_{1},q_{2}}$, summarized below.
\begin{thm}
\label{Thm1} For generic $q_{1}, q_{2}$ (such that $q_{1}^{a}q_{2}^{b}\neq 1$ for any $a,b \in \mathbb{N}$), $S_{q_{1},q_{2}}$ is generated by its  first graded component $S_{1}$. The ideal of relations has quadratic generators (from $S_{2}$) of the form
\begin{equation}
 \begin{multlined}
 \label{eq:(100)}
 q_{1}q_{2}(z^{m} *z^{n}+z^{n+2}*z^{m-2}) + z^{n}*z^{m} + z^{m-2}*z^{n+2} +\\ +(q_{1} + q_{2})(z^{m-1}*z^{n+1} + z^{n+1}* z^{m-1}) = 0
 \end{multlined}
  \end{equation}
  \end{thm}
\begin{proof}
The proof of the first assertion proceeds in $2$ steps. First, we show that the vector subspace
in $S_{n}$, generated by $\underbrace{S_{1} *\hdots * S_{1}}_{n}$ is closed under (ordinary) multiplication by symmetric Laurent polynomials, i.e. the vector spaced formed by the numerators of $\underbrace{S_{1} *\hdots * S_{1}}_{n}$ is an ideal in the ring of symmetric Laurent polynomials in $n$ variables. The next step is to show that
$\underbrace{S_{1} *\hdots * S_{1}}_{n}$ has no common zeros and therefore must coincide with $S_{n}$  (we use that  $S_{n}=\mathbb{C}[z_{1}^{\pm 1},\hdots,z_{n}^{\pm 1}]^{\mathfrak{S}_{n}}$ is finitely generated over $\mathbb{C}$ and, therefore, any proper ideal must be vanishing at some point).
\par
The first step is straightforward: if $F = \frac{f(z_{1},\hdots,z_{n})}{\triangle_{n}} =f_{1}(z)*\hdots *f_{n}(z) \in S_{1} *\hdots * S_{1}$ and $g(z_{1}, \hdots, z_{n})$ is a symmetric Laurent polynomial, i.e. $g(z_{1}, \hdots, z_{n})=\mbox{Sym}(z^{a_{1}}\hdots z^{a_{n}})$, then $gF=\frac{g(z_{1}, \hdots, z_{n})f(z_{1},\hdots,z_{n})}{\triangle_{n}}=f_{1}(z)z^{a_{1}}*\hdots *f_{n}(z)z^{a_{n}}$ also belongs to $\underbrace{S_{1} *\hdots * S_{1}}_{n}$. 
\par
Suppose that all functions from $S_{1}^{*n}$ vanish at a point with coordinates $(\alpha_{1},\hdots, \alpha_{n})$. We consider the collection of $n!$ functions $\delta_{\sigma}:= f^{\sigma}_{1}(z)*f^{\sigma}_{2}(z)*\hdots *f^{\sigma}_{n}(z) \in S_{1}^{*n}$, where $f^{\sigma}_{i}(z)=\prod\limits_{j\neq \sigma(i)}(z-\alpha_{j})$ and $\sigma \in S_{n}$. It is not hard to see that  $\delta_{\sigma}(\alpha_{1},\hdots, \alpha_{n})=c_{\sigma} \prod\limits_{i<j}\frac{ (\alpha_{\sigma(i)}-q_{1}\alpha_{\sigma(j)})(\alpha_{\sigma(i)}-q_{2}\alpha_{\sigma(j)})}{(\alpha_{\sigma(i)}-\alpha_{\sigma(j)})^{2}}$, with $c_{\sigma}\neq 0$. Therefore, we conclude that for a point $(\alpha_{1},\hdots, \alpha_{n})$  to be a common zero of $S_{1}^{*n}$ it is necessary  that all the summands of $1*\hdots *1 = \mbox{Alt}_{S_{n}}\left(\prod\limits_{i<j} \frac{(z_{i}-q_{1}z_{j})(z_{i}-q_{2}z_{j})}{(z_{i}-z_{j})^{2}}\right)$ vanish at this point. But if that was the case, in each of
the summands (corresponding to $\sigma$) there would exist a pair of indices $i<j$, such that either $z_{\sigma(i)} = q_{1}z_{\sigma(j)}$
or $z_{\sigma(i)} = q_{2}z_{\sigma(j)}$. Such a chain of equalities must form a cycle, i.e.
$z_{i_{1}}=q_{1,2}z_{i_{2}}$, $z_{i_{2}}=q_{1,2}z_{i_{3}},$ $\hdots z_{i_{k-1}}=q_{1,2}z_{i_{k}}$,  $z_{i_{k}}=q_{1,2}z_{i_{1}}$ (otherwise a summand with ordering of
$z_{i_{j}}$ reverse to the one in the chain exists and does not vanish), so $z_{i_{1}} = q_{1}^{a}
q_{2}^{b}
z_{i_{1}}$, which is impossible, since neither $z_{i_{1}}=0$ nor $ q_{1}^{a}
q_{2}^{b}=1$, a contradiction.
\par
It is direct to check that relations \eqref{eq:(100)} hold. We verify that the ideal of relations is generated by those in \eqref{eq:(100)}. 
 These relations  allow to rewrite every monomial $z^{i_{1}}* \hdots *z^{i_{k}}$ in such a way that
 \begin{equation}
 \label{monomials}
 i_{m+1}\leq i_{m}+1~~\forall m \in \{1,\hdots, k-1\}.
  \end{equation}
  This can be shown inductively. If there is only one term, the
 assertion is trivial. Otherwise, suppose there is $m$, s.t. $i_{m}\geq i_{m-1}+2$. We treat the cases $i_{m}\geq 
 i_{m-1}+3$ and $i_{m}=i_{m-1}+2$ separately. For the first one, using relation \eqref{eq:(100)}, we write $$z^{i_{m-1}}
 *z^{i_{m}}=-z^{i_{m}-2}*z^{i_{m-1}+2}-q_{1}q_{2}(z^{i_{m}}*z^{i_{m-1}}+z^{i_{m-1}+2}*z^{i_{m}-2})+$$ $$+(q_{1}+q_{2})
 (z^{i_{m}-1}*z^{i_{m-1}+1}+z^{i_{m-1}+1}*z^{i_{m}-1}).$$
\par
When $i_{m}=i_{m-1}+2$, this simplifies to $$z^{i_{m-1}}
 *z^{i_{m-1}+2}=-q_{1}q_{2}z^{i_{m-1}+2}*z^{i_{m-1}}+(q_{1}+q_{2})
 z^{i_{m-1}+1}*z^{i_{m-1}+1}.$$
 \par
 Clearly, in both cases all summands of the r.h.s. either satisfy the condition on the powers of $z$ or the difference $i_{m} - i_{m-1}$ become smaller.
\par
To show that the monomials, satisfying \eqref{monomials}, are linearly independent, we consider the specialization of the parameters to $q_{1}=0, q_{2}=1$. Indeed, any nontrivial $\mathbb{C}(q_{1},q_{2})$-linear relation between monomials $z^{i_{1}}* \hdots *z^{i_{k}}$ with
 $i_{m+1}\leq i_{m}+1~~\forall m \in \{1,\hdots, k\}$ can be multiplied by the l.c.m. of the denominators of the coefficients to induce a nontrivial relation for the specialization. We divide the coefficients by their g.c.d., hence, assume that the g.c.d=$1$.  
 \par
 To observe that a nontrivial relation is induced, we first specialize to $q_{2}=1$. If the relation became trivial, the coefficients of the initial relation (which are elements of $\mathbb{C}[q_{1},q_{2}]$) would be divisible by some power of $(q_{2}-1)$, but that contradicts the assumption that their g.c.d=$1$. Next, again, reduce to the case g.c.d=$1$. Specializing further to $q_{1}=0$, observe that the relation is still not vacuous, as otherwise all the coefficients (now elements of $\mathbb{C}[q_{1}]$) would be divisible by some power of $q_{1}$, hence, have g.c.d$\neq1$. However, $z^{i_{1}}* \hdots *z^{i_{k}}$ up to a constant becomes $\mbox{Alt}\left(\frac{z_{1}^{i_{1}+n-1}z_{2}^{i_{2}+n-2} \hdots z_{k}^{i_{k}}}{\bigtriangleup_{k}}\right)=\frac{\mbox{Sym}(z_{1}^{i_{1}+n-1}z_{2}^{i_{2}+n-2} \hdots z_{k}^{i_{k}})}{\bigtriangleup_{k}}$, with the new powers of $z_{j}$'s being nonicreasing. Therefore, using the fact that monomial symmetric Laurent polynomials are linearly independent, we see that the specializations of $z^{i_{1}}* \hdots *z^{i_{k}}$'s must be independent as well.   
\end{proof}

\section{Hilbert series of $S_{q_{1},q_{2}}/I$}
\par
We take the right ideal $I$ in $S_{q_{1},q_{2}}$, generated by $(z-\lambda)$ for some nonzero $\lambda \in \mathbb{C}$. To an $n$-tuple of points on the triangular grid (as on the picture below) we will associate a point in $(\mathbb{C}^{*})^{n}/S_{n}$ by taking the vertex with coordinates $(a,b)$ (the root vertex has coordinates $(0,0)$ and the point $(a,b)$ is $a$ steps down to the left and $b$ down to the right from the root) on the grid to coordinate $\lambda q_{1}^{a}q_{2}^{b}$ of the point. In particular, the root vertex gives coordinate $\lambda$. An example is provided below (the chosen subset of points is depicted in green). It will be shown that these points are precisely the common zero locus of all $F \in I_{n}$ ($n$th graded component of $I$). 

\par
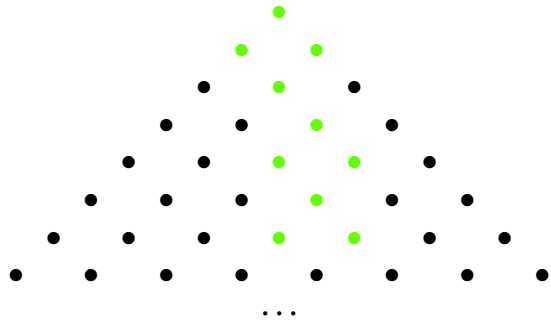
\begin{figure}[H]
	\label{FigI10}
\begin{center}
  \begin{tikzpicture}
    \node[] (1) at (0,0) {$\color{green!60!yellow}{\bullet}$};
    \node[] (2) at (-0.5,-0.5)  {$\color{green!60!yellow}{\bullet}$};
     \node[] (2) at (-1,-1)  {$\bullet$};
      \node[] (2) at (1, -1)  {$\bullet$};
    \node[] (3) at (0.5,-0.5) {$\color{green!60!yellow}{\bullet}$};
     \node[] (4) at (0,-1) {$\color{green!60!yellow}{\bullet}$};
       \node[] (2) at (-1.5,-1.5)  {$\bullet$};
      \node[] (2) at (-0.5, -1.5)  {$\bullet$};
    \node[] (3) at (0.5,-1.5) {$\color{green!60!yellow}{\bullet}$};
    \node[] (3) at (1.5,-1.5) {$\bullet$};
    
       \node[] (2) at (-2,-2)  {$\bullet$};
      \node[] (2) at (-1, -2)  {$\bullet$};
    \node[] (3) at (0,-2) {$\color{green!60!yellow}{\bullet}$};
    \node[] (3) at (1,-2) {$\color{green!60!yellow}{\bullet}$};
    \node[] (3) at (2,-2) {$\bullet$};
    
       \node[] (2) at (-2.5,-2.5)  {$\bullet$};
      \node[] (2) at (-1.5, -2.5)  {$\bullet$};
    \node[] (3) at (0.5,-2.5) {$\color{green!60!yellow}{\bullet}$};
    \node[] (3) at (1.5,-2.5) {$\bullet$};
    \node[] (3) at (2.5,-2.5) {$\bullet$};
    \node[] (3) at (-0.5,-2.5) {$\bullet$};
    
      \node[] (2) at (-3, -3)  {$\bullet$};
    \node[] (3) at (-2,-3) {$\bullet$};
       \node[] (2) at (-1, -3)  {$\bullet$};
    \node[] (3) at (0,-3) {$\color{green!60!yellow}{\bullet}$};
    \node[] (3) at (1,-3) {$\color{green!60!yellow}{\bullet}$};
    \node[] (3) at (2,-3) {$\bullet$};
      \node[] (2) at (3, -3)  {$\bullet$};
      
       \node[] (2) at (-3.5,-3.5)  {$\bullet$};
           \node[] (2) at (-2.5,-3.5)  {$\bullet$};
      \node[] (2) at (-1.5, -3.5)  {$\bullet$};
    \node[] (3) at (0.5,-3.5) {$\bullet$};
    \node[] (3) at (1.5,-3.5) {$\bullet$};
    \node[] (3) at (2.5,-3.5) {$\bullet$};
    \node[] (3) at (-0.5,-3.5) {$\bullet$};
     \node[] (2) at (3.5,-3.5)  {$\bullet$};
     
       \node[] (2) at (0,-4)  {$\hdots$};
      \end{tikzpicture}
      \end{center}
    \caption{point $(\lambda,q_{1}\lambda,q_{2}\lambda,q_{1}q_{2}\lambda, q_{1}q_{2}^{2}\lambda, q_{1}^{2}q_{2}^{2}\lambda,  q_{1}^{3}q_{2}\lambda, q_{1}^{3}q_{2}^{2}\lambda, q_{1}^{3}q_{2}^{3}\lambda,q_{1}^{4}q_{2}^{2}\lambda)$}
    \end{figure}
\subsection{Generic parameters}
\par
The Hilbert series of $S_{q_{1},q_{2}}/I$ are given by the following result.

\par
\begin{thm}
\label{Thm2}
Assume a subset of vertices on the grid is chosen and $(0,0)$ belongs to the subset. The number of rooted binary trees supported on the vertices from the subset is either
zero or $2^{k}$, where $k$ is the number of vertices $(a; b)$ with neighbors
$(a-1; b)$ and $(a; b-1)$ also included in the subset (we denote this set of
vertices by $X$). The multiplicity of the point associated to the subset is also $2^{k}$.
\end{thm}

\begin{proof}
The claim about the number of binary trees supported on the given points holds, since constructing trees by connecting vertices from bottom to top row, the only ambiguity comes from the vertices in $X$ and there are precisely $2^{k}$ possible ways to draw the tree.
\par
We first notice that any element $G(z_{1},\hdots, z_{n})$ of the $n$th graded component of the ideal ($I_{n}$) is of the form $(z-\lambda)*F(z_{1},\hdots, z_{n-1})$ with $F(z_{1},\hdots, z_{n-1}) \in S_{n-1}$. More explicitly, we can write $G(z_{1},\hdots, z_{n})=\mbox{Alt}_{S_{n}}\left(F(z_{2}, \hdots, z_{n})(z_{1}-\lambda)\prod\limits_{i\in \{2,\hdots n\}}\mu(z_{1},z_{i})\right)$. Now for a point to be a common zero of all rational functions in $I_{n}$, all summands of \begin{equation}
 \label{eq:(2)}\mbox{Alt}_{S_{n}}\left((z_{1}-\lambda)\prod\limits_{i\in \{2,\hdots n\}}\mu(z_{1},z_{i})\right)
 \end{equation}
  must vanish at this point. For $\sigma \in S_{n}$ the corresponding summand is $(z_{\sigma(1)}-\lambda)\prod\limits_{i\in \{2,\hdots n\}}\mu(z_{\sigma(1)},z_{\sigma(i)})=(z_{\sigma(1)}-\lambda)\prod\limits_{i\neq \sigma(1)}\mu(z_{\sigma(1)},z_{i})$. It vanishes at a point if either $z_{\sigma(1)}=\lambda$ or $z_{\sigma(1)}=q_{1,2}z_{i}$. As $(z_{1}, \hdots, z_{n})$ and $(z_{h(1)}, \hdots, z_{h(n)})$ represent the same point for any $h \in S_{n}$ and for generic $q_{1}, q_{2}$ we are not allowed to have cycles, i.e. chains of equalities of type $z_{i_{1}}=q_{1,2}z_{i_{2}}$, $z_{i_{2}}=q_{1,2}z_{i_{3}}$ $\hdots z_{i_{k-1}}=q_{1,2}z_{i_{k}}$,  $z_{i_{k}}=q_{1,2}z_{i_{1}}$, for some representative of the point we must have $z_{1}=\lambda$, $z_{2}=q_{1,2}z_{1}=q_{1,2}\lambda$, etc. In other words, the representative of the point at which every element of $I_{n}$ vanishes can be chosen so, that $z_{1}=\lambda$, $z_{i\geq 2}=q_{1,2}z_{j\geq 2}$ for some $j<i$ and all $z_{i}$'s are distinct.
\par
At some points not only the rational function \eqref{eq:(2)} is zero, but some (linear combinations) of its partial and higher order derivatives (linear combinations of partial derivatives and higher order derivatives arise separately) vanish as well. It is not hard to see that this happens if there are vertices $(a; b)$ with both neighbors
$(a-1; b)$ and $(a; b-1)$ in the subset of the grid, which the coordinates of the point are derived from. The simplest example is the point $(\lambda,q_{1}\lambda,q_{2}\lambda,q_{1}q_{2}\lambda)$, where both $L(\lambda,q_{1}\lambda,q_{2}\lambda,q_{1}q_{2}\lambda)=L'_{z_{4}}(\lambda,q_{1}\lambda,q_{2}\lambda,q_{1}q_{2}\lambda)$ are zero for any $L(z_{1}, z_{2}, z_{3}, z_{4}) \in I_{4}$.
\par
We choose the element $(a_{1},b_{1})$ of $X$ that corresponds to $z_{(a_{1},b_{1})}$ with the smallest index. Then the derivatives with respect to $z_{(a_{1},b_{1})}$ of all summands of  \eqref{eq:(2)} corresponding to $\sigma \in S_{n}$ with $z_{\sigma(1)}\neq q_{1,2}z_{(a_{1},b_{1})}$ vanish at the point (one can use the product rule for each of the summands and see that all summands produced vanish). Thus, if none of $(a_{1}+1,b_{1})$, $(a_{1},b_{1}+1)$ belong to the set-theoretical difference of the chosen subset of the grid with $X$, the derivative with respect to $z_{i_{1}}$ vanishes. Otherwise, one (or both) summands $$-q_{1}(z_{(a_{1}+1,b_{1})}-q_{2}z_{(a_{1},b_{1})})(z_{(a_{1}+1,b_{1})}-\lambda)\prod\limits_{i\neq (a_{1}+1,b_{1})}(z_{(a_{1}+1,b_{1})}-q_{1}z_{i})(z_{(a_{1}+1,b_{1})}-q_{2}z_{i}),$$ $$-q_{2}(z_{(a_{1},b_{1}+1)}-q_{1}z_{(a_{1},b_{1})})(z_{(a_{1},b_{1}+1)}-\lambda)\prod\limits_{i\neq (a_{1}+1,b_{1})}(z_{(a_{1},b_{1}+1)}-q_{1}z_{i})(z_{(a_{1},b_{1}+1)}-q_{2}z_{i})$$ is (are) the only nonzero terms, when we evaluate the derivative of \eqref{eq:(2)} with respect to $z_{(a_{1},b_{1})}$ at the point corresponding to the vertices on the diagram. However, the same is true for the derivative(s) with respect to either of $z_{(a_{1}+1,b_{1})}$, $z_{(a_{1},b_{1}+1)}$ or both (the only nonzero terms being $(z_{(a_{1}+1,b_{1})}-q_{2}z_{(a_{1},b_{1})})(z_{(a_{1}+1,b_{1})}-\lambda)\prod\limits_{i\neq (a_{1}+1,b_{1})}(z_{(a_{1}+1,b_{1})}-q_{1}z_{i})(z_{(a_{1}+1,b_{1})}-q_{2}z_{i})$; $(z_{(a_{1},b_{1}+1)}-q_{1}z_{(a_{1},b_{1})})(z_{(a_{1},b_{1}+1)}-\lambda)\prod\limits_{i\neq (a_{1}+1,b_{1})}(z_{(a_{1},b_{1}+1)}-q_{1}z_{i})(z_{(a_{1},b_{1}+1)}-q_{2}z_{i})$) and adding $q_{1}$-multiple of the derivative with respect to $z_{(a_{1}+1,b_{1})}$ or $q_{2}$-multiple of the derivative with respect to $z_{(a_{1},b_{1}+1)}$ or both resolves the issue if there are no vertices one step below to the left or right of $(a_{1}+1,b_{1})$, $(a_{1},b_{1}+1)$, belonging to the set-theoretical difference of the chosen subset of the grid with $X$. In case there are such vertices we add derivatives with respect to those with corresponding coefficients. Similarly, if we have more than one linear relation between first-order derivatives at our point, the relations between second order derivatives are produced by taking 'mixed derivatives, as suggested by the first order relations', etc. So, the rank of the system of linear relations obtained is exactly $2^{k}$ (number of subsets of $X$). 
\end{proof}
\par
An example is provided below.
\par
\textbf{Example}. Let $F(z_{1},\hdots,z_{10}) \in I_{10}$, then $F(\lambda,q_{1}\lambda,q_{2}\lambda,\hdots, q_{1}^{3}q_{2}^{4}\lambda)=0$. There are two linear relations between first order derivatives: 
 \begin{flalign*}
&F'_{z_{4}}(\lambda,q_{1}\lambda,q_{2}\lambda,\hdots, q_{1}^{3}q_{2}^{4}\lambda)+q_{2}F'_{z_{5}}(\lambda,q_{1}\lambda,q_{2}\lambda,\hdots, q_{1}^{3}q_{2}^{4}\lambda)+\\
&+q_{1}F'_{z_{6}}(\lambda,q_{1}\lambda,q_{2}\lambda,\hdots, q_{1}^{3}q_{2}^{4}\lambda)+q_{2}F'_{z_{7}}(\lambda,q_{1}\lambda,q_{2}\lambda,\hdots, q_{1}^{3}q_{2}^{4}\lambda)=0 \mbox{ and }\\
&F'_{z_{8}}(\lambda,q_{1}\lambda,q_{2}\lambda,\hdots, q_{1}^{3}q_{2}^{4}\lambda)+q_{1}F'_{z_{9}}(\lambda,q_{1}\lambda,q_{2}\lambda,\hdots, q_{1}^{3}q_{2}^{4}\lambda)+q_{2}F'_{z_{10}}(\lambda,q_{1}\lambda,q_{2}\lambda,\hdots, q_{1}^{3}q_{2}^{4}\lambda)=0.
\end{flalign*}

\par
Finally, there is one linear relation between second order derivatives: 
 \begin{flalign*}
&F''_{z_{4}z_{8}}(\lambda,q_{1}\lambda,q_{2}\lambda,\hdots, q_{1}^{3}q_{2}^{4}\lambda)+q_{2}F''_{z_{5}z_{8}}(\lambda,q_{1}\lambda,q_{2}\lambda,\hdots, q_{1}^{3}q_{2}^{4}\lambda)+q_{1}F''_{z_{6}z_{8}}(\lambda,q_{1}\lambda,q_{2}\lambda,\hdots, q_{1}^{3}q_{2}^{4}\lambda)+\\
&+q_{2}F''_{z_{7}z_{8}}(\lambda,q_{1}\lambda,q_{2}\lambda,\hdots, q_{1}^{3}q_{2}^{4}\lambda)+q_{1}F''_{z_{4}z_{9}}(\lambda,q_{1}\lambda,q_{2}\lambda,\hdots, q_{1}^{3}q_{2}^{4}\lambda)+q_{1}q_{2}F''_{z_{5}z_{9}}(\lambda,q_{1}\lambda,q_{2}\lambda,\hdots, q_{1}^{3}q_{2}^{4}\lambda)+\\
&+q_{1}^{2}F''_{z_{6}z_{9}}(\lambda,q_{1}\lambda,q_{2}\lambda,\hdots, q_{1}^{3}q_{2}^{4}\lambda)+q_{1}q_{2}F''_{z_{7}z_{9}}(\lambda,q_{1}\lambda,q_{2}\lambda,\hdots, q_{1}^{3}q_{2}^{4}\lambda)+q_{2}F''_{z_{4}z_{10}}(\lambda,q_{1}\lambda,q_{2}\lambda,\hdots, q_{1}^{3}q_{2}^{4}\lambda)+\\
&+q_{2}^{2}F''_{z_{5}z_{10}}(\lambda,q_{1}\lambda,q_{2}\lambda,\hdots, q_{1}^{3}q_{2}^{4}\lambda)+q_{1}q_{2}F''_{z_{6}z_{10}}(\lambda,q_{1}\lambda,q_{2}\lambda,\hdots, q_{1}^{3}q_{2}^{4}\lambda)+q_{2}^{2}F''_{z_{7}z_{10}}(\lambda,q_{1}\lambda,q_{2}\lambda,\hdots, q_{1}^{3}q_{2}^{4}\lambda)\\
&=0. 
\end{flalign*}
\par
Thus we conclude that the point $(\lambda,q_{1}\lambda,q_{2}\lambda,\hdots, q_{1}^{3}q_{2}^{4}\lambda)$ must be taken with multiplicity four. This point corresponds to the 'green part' of the picture on Figure \ref{FigI10} above.

\par
\begin{cor}
\label{Cor2}
The dimension of $(S_{q_{1},q_{2}}/I)_{n}$ is equal to the number of  subsets of vertices on the grid of cardinality $n$ containing the root $(0,0)$, where each subset is counted with multiplicity determined in theorem \ref{Thm2}. Furthermore, the dimensions are bounded above by dim$(S_{q_{1},q_{2}}/I)_{n} \leq c_{n}$, where $c_{n}=\frac{{2n \choose n}}{n+1}$ is the $n$th Catalan number.
\end{cor}
\par
The inequalities for dimensions above are strict for $n\geq 5$. To be more precise, the binary trees, which do not correspond to points, where all polynomials from $I_{n}$ vanish, are the trees, which have at least one pair of vertices with equal coordinates on the grid. The only such tree with $n=5$ vertices appears on Figure ~\ref{bad tree} and eight trees with $n=6$ vertices on Figure ~\ref{bad trees} below.

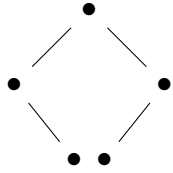
\begin{figure}[H]

\begin{center}
  \begin{tikzpicture}
    \node[] (1) at (0,0) {$\bullet$};
    \node[] (3) at (2,0) {$\bullet$};
    
    \node[] (12) at (1,1)  {$\bullet$};
    \node[] (4) at (0.8,-1)  {$\bullet$};
       \node[] (5) at (1.2,-1)  {$\bullet$};

      \path[]
    (1) edge node {} (12)
    (12) edge node {} (3)
    (3) edge node {} (5)
    (1) edge node {} (4);
 
\end{tikzpicture}
\end{center}
\caption{Forbidden Tree of Degree $5$}
\label{bad tree}
\end{figure}
\par
\begin{figure}[H]
\begin{tikzpicture}
   \node[] (-1) at (1,1) {$\bullet$};
    \node[] (1) at (1,-1) {$\bullet$};
    \node[] (3) at (3,-1) {$\bullet$};
     \node[] (12) at (2,0)  {$\bullet$};
    \node[] (4) at (1.8,-2)  {$\bullet$};
       \node[] (5) at (2.2,-2)  {$\bullet$};
\path[]
      (-1) edge node {} (12)
    (1) edge node {} (12)
    (12) edge node {} (3)
    (3) edge node {} (5)
    (1) edge node {} (4);
 \end{tikzpicture}
  \begin{tikzpicture}
   \node[] (-1) at (2,2) {$\bullet$};
    \node[] (1) at (0,0) {$\bullet$};
    \node[] (3) at (2,0) {$\bullet$};
    \node[] (12) at (1,1)  {$\bullet$};
    \node[] (4) at (0.8,-1)  {$\bullet$};
       \node[] (5) at (1.2,-1)  {$\bullet$};
\path[]
      (-1) edge node {} (12)
    (1) edge node {} (12)
    (12) edge node {} (3)
    (3) edge node {} (5)
    (1) edge node {} (4);
 \end{tikzpicture}
  \begin{tikzpicture}
   \node[] (-1) at (0,-1) {$\bullet$};
    \node[] (1) at (1,0) {$\bullet$};
    \node[] (3) at (3,0) {$\bullet$};
    \node[] (12) at (2,1)  {$\bullet$};
    \node[] (4) at (1.8,-1)  {$\bullet$};
       \node[] (5) at (2.2,-1)  {$\bullet$};

      \path[]
      (-1) edge node {} (1)
    (1) edge node {} (12)
    (12) edge node {} (3)
    (3) edge node {} (5)
    (1) edge node {} (4);
 
\end{tikzpicture}
  \begin{tikzpicture}
   \node[] (-1) at (4,0) {$\bullet$};
    \node[] (1) at (1,1) {$\bullet$};
    \node[] (3) at (3,1) {$\bullet$};
    
    \node[] (12) at (2,2)  {$\bullet$};
    \node[] (4) at (1.8,0)  {$\bullet$};
       \node[] (5) at (2.2,0)  {$\bullet$};

      \path[]
      (-1) edge node {} (3)
    (1) edge node {} (12)
    (12) edge node {} (3)
    (3) edge node {} (5)
    (1) edge node {} (4);
 \end{tikzpicture}
\end{figure}
\par

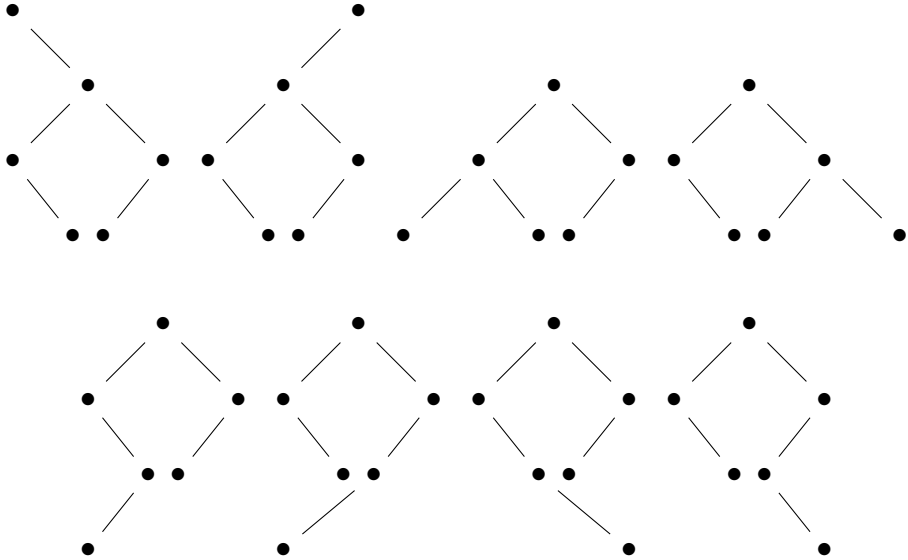
\begin{figure}[h]
  \begin{tikzpicture}
   \node[] (-1) at (1,-2) {$\bullet$};
    \node[] (1) at (1,0) {$\bullet$};
    \node[] (3) at (3,0) {$\bullet$};
    
    \node[] (12) at (2,1)  {$\bullet$};
    \node[] (4) at (1.8,-1)  {$\bullet$};
       \node[] (5) at (2.2,-1)  {$\bullet$};

      \path[]
      (-1) edge node {} (4)
    (1) edge node {} (12)
    (12) edge node {} (3)
    (3) edge node {} (5)
    (1) edge node {} (4);
 
\end{tikzpicture}
 \begin{tikzpicture}
   \node[] (-1) at (1,-2) {$\bullet$};
    \node[] (1) at (1,0) {$\bullet$};
    \node[] (3) at (3,0) {$\bullet$};
    
    \node[] (12) at (2,1)  {$\bullet$};
    \node[] (4) at (1.8,-1)  {$\bullet$};
       \node[] (5) at (2.2,-1)  {$\bullet$};
        \path[]
      (-1) edge node {} (5)
    (1) edge node {} (12)
    (12) edge node {} (3)
    (3) edge node {} (5)
    (1) edge node {} (4);
\end{tikzpicture}
 \begin{tikzpicture}
   \node[] (-1) at (3,-2) {$\bullet$};
    \node[] (1) at (1,0) {$\bullet$};
    \node[] (3) at (3,0) {$\bullet$};
    
    \node[] (12) at (2,1)  {$\bullet$};
    \node[] (4) at (1.8,-1)  {$\bullet$};
       \node[] (5) at (2.2,-1)  {$\bullet$};

      \path[]
      (-1) edge node {} (4)
    (1) edge node {} (12)
    (12) edge node {} (3)
    (3) edge node {} (5)
    (1) edge node {} (4);
 \end{tikzpicture}
 \begin{tikzpicture}
   \node[] (-1) at (3,-2) {$\bullet$};
    \node[] (1) at (1,0) {$\bullet$};
    \node[] (3) at (3,0) {$\bullet$};
    
    \node[] (12) at (2,1)  {$\bullet$};
    \node[] (4) at (1.8,-1)  {$\bullet$};
       \node[] (5) at (2.2,-1)  {$\bullet$};

      \path[]
      (-1) edge node {} (5)
    (1) edge node {} (12)
    (12) edge node {} (3)
    (3) edge node {} (5)
    (1) edge node {} (4);
 \end{tikzpicture}
\caption{$8$ Forbidden Trees of Degree $6$}
\label{bad trees}
\end{figure}
\subsection{Nongeneric parameters $(q_{1}^{a}q_{2}^{b}=1)$}
Here we assume that $q_{1}$ and $q_{2}$ are not generic, i.e. there exist $a,b \in \mathbb{N}$, s.t. $q_{1}^{a}q_{2}^{b}=1$. One of the significant changes is that $S_{1}$ does not generate $S_{q_{1},q_{2}}$ anymore. For example, if $a=2$, $b=1$, then every element of $S_{1}*S_{1}*S_{1}$ is zero at any point with coordinates $(\alpha,q_{1}\alpha,q_{1}q_{2}\alpha) ~~\forall \alpha \in \mathbb{C}^{*}$. 
\par
Therefore, counting the dimensions of $(S/I)_{n}$ does not make sense - they will become infinite, starting with $n=a+b$. The reasonable modification to what we have done so far is to work with the subalgebra of $S_{q_{1},q_{2}}$ generated by $S_{1}$, which will be denoted by $\widetilde{S}_{q_{1},q_{2}}$, consider the right ideal generated by $(z-\lambda)$ and calculate the dimensions of $(\widetilde{S}_{q_{1},q_{2}}/I)_{n}$. 
\begin{rmk}
In case $q_{1}$ and $q_{2}$ were generic $\widetilde{S}_{q_{1},q_{2}}$ and $S_{q_{1},q_{2}}$ were isomorphic, so there was no difference in working with $(S/I)_{n}$ or $(\widetilde{S}_{q_{1},q_{2}}/I)_{n}$.
\end{rmk}

\par
\begin{thm}
\label{Thm3}
Assume a subset of vertices on the grid is chosen, $(0,0)$ belongs to the subset and none of the points in the region $R:=\{(x,y)|x\geq a, y \geq b\}$ are included in the subset. The number of binary trees supported on the vertices from the subset is either zero or $2^{k}$, where $k$ is the number of vertices $(c; d)$ with neighbors
$(c-1; d)$ and $(c; d-1)$ also included in the subset (we denote this set of
vertices by $X$). The multiplicity of the point associated to the subset is also $2^{k}$. 
\end{thm}
\begin{figure}[H]
\begin {center}
 \begin{tikzpicture}
    \node[] (1) at (0,0) {$\bullet$};
    \node[] (2) at (-0.5,-0.5)  {$\bullet$};
     \node[] (2) at (-1,-1)  {$\bullet$};
      \node[] (2) at (1, -1)  {$\bullet$};
    \node[] (3) at (0.5,-0.5) {$\bullet$};
     \node[] (4) at (0,-1) {$\bullet$};
       \node[] (2) at (-1.5,-1.5)  {$\bullet$};
      \node[] (2) at (-0.5, -1.5)  {$\bullet$};
    \node[] (3) at (0.5,-1.5) {$\red{(a,b)}$};
    \node[] (3) at (1.5,-1.5) {$\bullet$};
    
       \node[] (2) at (-2,-2)  {$\bullet$};
      \node[] (2) at (-1, -2)  {$\bullet$};
    \node[] (3) at (0,-2) {$\red{\bullet}$};
    \node[] (3) at (1,-2) {$\red{\bullet}$};
    \node[] (3) at (2,-2) {$\bullet$};
    
       \node[] (2) at (-2.5,-2.5)  {$\bullet$};
      \node[] (2) at (-1.5, -2.5)  {$\bullet$};
    \node[] (3) at (0.5,-2.5) {$\red{\bullet}$};
    \node[] (3) at (1.5,-2.5) {$\red{\bullet}$};
    \node[] (3) at (2.5,-2.5) {$\bullet$};
    \node[] (3) at (-0.5,-2.5) {$\red{\bullet}$};
    
      \node[] (2) at (0,-3)  {$\hdots$};
      \end{tikzpicture}
      \end{center}
    \caption{Restriction on subset of vertices}
    \label{example}
\end{figure}
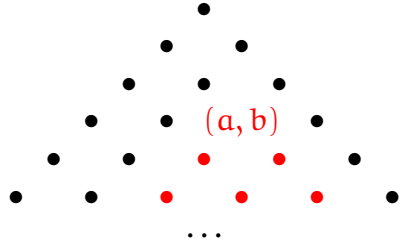
\par
\begin{proof}
One can argue analogously to the proof of theorem \ref{Thm2}, the only difference being that any point associated to a subset of the grid, with nonempty intersection with $R$ (the region depicted in red on the picture above) has equal coordinates, thus, can not be taken into consideration, since any function from $(\widetilde{S}_{q_{1},q_{2}})_{n}$ has a pole at such point.  
\end{proof}

\par
\begin{cor}
\label{Cor3}
The dimension of $(\widetilde{S}_{q_{1},q_{2}}/I)_{n}$ is equal to the number of  subsets of vertices on the grid of cardinality $n$ containing the root $(0,0)$ and no points from $R$, where each subset is counted with multiplicity determined in theorem \ref{Thm3}.
\end{cor}

\section{Hilbert series of a representation.}
The result of this part is very similar to the one established in Theorem $4$ in the Appendix of \cite{Dots}. However, for the case relevant to us, the machinery used in \cite{Dots} is not required, and a more elementary purely combinatorial approach is sufficient to provide a proof of the result.  Throughout this section we assume that the parameters $q_{1}$ and $q_{2}$ are generic and commute. Let $S_{q_{1},q_{2}}^{+}$ denote the subalgebra generated by $z^{i}, i\geq0$ and $S_{1}^{+}$ its first graded
component. We introduce the representation $V := \mbox{Ind}_{S_{1}^{+}}^{S_{q_{1},q_{2}}^{+}}\mathbb{C}_{1}$ of $S_{q_{1},q_{2}}^{+}$, where $\mathbb{C}_{1}$ is the one dimensional representation of $S_{1}^{+}$ with $z^{i>0}v=0$ and $z^{0}v=v$ for $v$ that spans $\mathbb{C}_{1}$.
\par
\begin{thm}
\label{Thm6} The dimension of $(S^{+}_{q_{1},q_{2}})_{n}v$ is equal to the number of Dyck paths of length $2n$, which is known to be the corresponding Catalan number. 
\end{thm}
\begin{proof}
Here it is convenient to use the presentation of $S_{q_{1},q_{2}}$ in terms of generators and relations (see theorem \ref{Thm1}). This presentation allows to choose a basis of $(S^{+}_{q_{1},q_{2}})_{n}$, consisting of $z^{i_{1}} \hdots z^{i_{n}}$, with  $i_{m+1} \leq  i_{m}+1 ~~\forall m \in \{1,\hdots,n-1\}, i_{j}\geq 0 ~~\forall j \in \{1,\hdots,n\}$. The monomials that act nontrivially on $v$ have an additional condition that $i_{1}=0$. Now to the vector $vz^{i_{1}} \hdots z^{i_{n}}$ with  $z^{i_{1}} \hdots z^{i_{n}}$ as above we associate the Dyck path according to the following three-step procedure, (it is not hard to see, that this is invertible, hence, gives a bijection): 
\begin{enumerate}
	\item subdivide the monomial into monomials with strictly increasing indices;
	\item the starting point for the part of the path corresponding to each monomial obtained on the previous step is the
	terminal point for the part of the path associated to the previous monomial, the path goes the degree of the monomial blocks up and then down until the $y$-coordinate is equal to the index of the first
	variable in the next monomial;
	\item if the $y$-coordinate of the terminal point of the path is greater than $0$, complete the path by adding $y$ steps down.
\end{enumerate}
\end{proof}
\par

\textbf{Example}. The pictures below  illustrate the correspondence for $n=3$.
\begin{figure}[H]
\label{Fig5}
\begin{center}
  \begin{tikzpicture}
    \node[] (1) at (0,0) {$\bullet$};
    \node[] (2) at (1,0)  {$\bullet$};
    \node[] (3) at (2,0) {$\bullet$};
       \node[] (4) at (3,0)  {$\bullet$};
    \node[] (5) at (4,0) {$\bullet$};
       \node[] (6) at (5,0)  {$\bullet$};
    \node[] (7) at (6,0) {$\bullet$};
    
        \node[] (11) at (0,1) {$\bullet$};
    \node[] (12) at (1,1)  {$\bullet$};
    \node[] (13) at (2,1) {$\bullet$};
       \node[] (14) at (3,1)  {$\bullet$};
    \node[] (15) at (4,1) {$\bullet$};
       \node[] (16) at (5,1)  {$\bullet$};
    \node[] (17) at (6,1) {$\bullet$};

      \path[]
    (1) edge node {} (12)
    (12) edge node {} (3)
      (3) edge node {} (14)
    (14) edge node {} (5)
      (5) edge node {} (16)
    (16) edge node {} (7) ;
 
\end{tikzpicture}
\end{center}
\caption{$z^{0}\red{|}$ $z^{0}\red{|}$ $z^{0}$}
\end{figure}

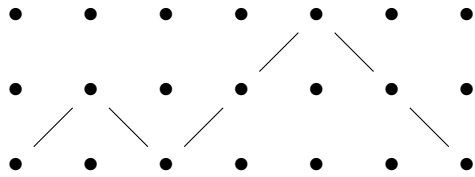
\begin{figure}[H]
\begin{center}
  \begin{tikzpicture}
    \node[] (1) at (0,0) {$\bullet$};
    \node[] (2) at (1,0)  {$\bullet$};
    \node[] (3) at (2,0) {$\bullet$};
       \node[] (4) at (3,0)  {$\bullet$};
    \node[] (5) at (4,0) {$\bullet$};
       \node[] (6) at (5,0)  {$\bullet$};
    \node[] (7) at (6,0) {$\bullet$};
    
        \node[] (11) at (0,1) {$\bullet$};
    \node[] (12) at (1,1)  {$\bullet$};
    \node[] (13) at (2,1) {$\bullet$};
       \node[] (14) at (3,1)  {$\bullet$};
    \node[] (15) at (4,1) {$\bullet$};
       \node[] (16) at (5,1)  {$\bullet$};
    \node[] (17) at (6,1) {$\bullet$};
    
        \node[] (21) at (0,2) {$\bullet$};
    \node[] (22) at (1,2)  {$\bullet$};
    \node[] (23) at (2,2) {$\bullet$};
       \node[] (24) at (3,2)  {$\bullet$};
    \node[] (25) at (4,2) {$\bullet$};
       \node[] (26) at (5,2)  {$\bullet$};
    \node[] (27) at (6,2) {$\bullet$};
    
    \path[]
    (1) edge node {} (12)
    (12) edge node {} (3)
      (3) edge node {} (14)
    (14) edge node {} (25)
      (25) edge node {} (16)
    (16) edge node {} (7) ;
  
\end{tikzpicture}
\end{center}
   \caption{$z^{0}\red{|}$ $z^{0}z^{1}$}
\end{figure}

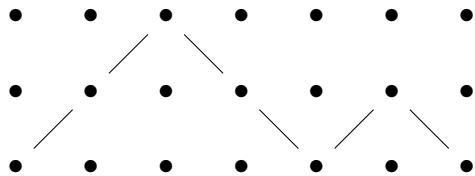
\begin{figure}[H]
\begin{center}
  \begin{tikzpicture}
    \node[] (1) at (0,0) {$\bullet$};
    \node[] (2) at (1,0)  {$\bullet$};
    \node[] (3) at (2,0) {$\bullet$};
       \node[] (4) at (3,0)  {$\bullet$};
    \node[] (5) at (4,0) {$\bullet$};
       \node[] (6) at (5,0)  {$\bullet$};
    \node[] (7) at (6,0) {$\bullet$};
    
        \node[] (11) at (0,1) {$\bullet$};
    \node[] (12) at (1,1)  {$\bullet$};
    \node[] (13) at (2,1) {$\bullet$};
       \node[] (14) at (3,1)  {$\bullet$};
    \node[] (15) at (4,1) {$\bullet$};
       \node[] (16) at (5,1)  {$\bullet$};
    \node[] (17) at (6,1) {$\bullet$};
    
        \node[] (21) at (0,2) {$\bullet$};
    \node[] (22) at (1,2)  {$\bullet$};
    \node[] (23) at (2,2) {$\bullet$};
       \node[] (24) at (3,2)  {$\bullet$};
    \node[] (25) at (4,2) {$\bullet$};
       \node[] (26) at (5,2)  {$\bullet$};
    \node[] (27) at (6,2) {$\bullet$};
    
    \path[]
    (1) edge node {} (12)
    (12) edge node {} (23)
      (23) edge node {} (14)
    (14) edge node {} (5)
      (5) edge node {} (16)
    (16) edge node {} (7) ;
    
\end{tikzpicture}
\end{center}
 \caption{$z^{0}z^{1}\red{|}$ $z^{0}$}
 
\end{figure}
\begin{figure}[H]
\begin{center}
  \begin{tikzpicture}
    \node[] (1) at (0,0) {$\bullet$};
    \node[] (2) at (1,0)  {$\bullet$};
    \node[] (3) at (2,0) {$\bullet$};
       \node[] (4) at (3,0)  {$\bullet$};
    \node[] (5) at (4,0) {$\bullet$};
       \node[] (6) at (5,0)  {$\bullet$};
    \node[] (7) at (6,0) {$\bullet$};
    
        \node[] (11) at (0,1) {$\bullet$};
    \node[] (12) at (1,1)  {$\bullet$};
    \node[] (13) at (2,1) {$\bullet$};
       \node[] (14) at (3,1)  {$\bullet$};
    \node[] (15) at (4,1) {$\bullet$};
       \node[] (16) at (5,1)  {$\bullet$};
    \node[] (17) at (6,1) {$\bullet$};
    
        \node[] (21) at (0,2) {$\bullet$};
    \node[] (22) at (1,2)  {$\bullet$};
    \node[] (23) at (2,2) {$\bullet$};
       \node[] (24) at (3,2)  {$\bullet$};
    \node[] (25) at (4,2) {$\bullet$};
       \node[] (26) at (5,2)  {$\bullet$};
    \node[] (27) at (6,2) {$\bullet$};
    
    \path[]
    (1) edge node {} (12)
    (12) edge node {} (23)
      (23) edge node {} (14)
    (14) edge node {} (25)
      (25) edge node {} (16)
    (16) edge node {} (7) ;
\end{tikzpicture}
\caption{$z^{0}z^{1}\red{|}$ $z^{1}$}
\end{center}
\end{figure}
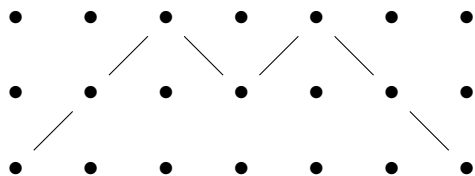

\begin{figure}[H]
	\label{Fig9}
\begin{center}
  \begin{tikzpicture}
    \node[] (1) at (0,0) {$\bullet$};
    \node[] (2) at (1,0)  {$\bullet$};
    \node[] (3) at (2,0) {$\bullet$};
       \node[] (4) at (3,0)  {$\bullet$};
    \node[] (5) at (4,0) {$\bullet$};
       \node[] (6) at (5,0)  {$\bullet$};
    \node[] (7) at (6,0) {$\bullet$};
    
        \node[] (11) at (0,1) {$\bullet$};
    \node[] (12) at (1,1)  {$\bullet$};
    \node[] (13) at (2,1) {$\bullet$};
       \node[] (14) at (3,1)  {$\bullet$};
    \node[] (15) at (4,1) {$\bullet$};
       \node[] (16) at (5,1)  {$\bullet$};
    \node[] (17) at (6,1) {$\bullet$};
    
        \node[] (21) at (0,2) {$\bullet$};
    \node[] (22) at (1,2)  {$\bullet$};
    \node[] (23) at (2,2) {$\bullet$};
       \node[] (24) at (3,2)  {$\bullet$};
    \node[] (25) at (4,2) {$\bullet$};
       \node[] (26) at (5,2)  {$\bullet$};
    \node[] (27) at (6,2) {$\bullet$};
    
          \node[] (31) at (0,3) {$\bullet$};
    \node[] (32) at (1,3)  {$\bullet$};
    \node[] (33) at (2,3) {$\bullet$};
       \node[] (34) at (3,3)  {$\bullet$};
    \node[] (35) at (4,3) {$\bullet$};
       \node[] (36) at (5,3)  {$\bullet$};
    \node[] (37) at (6,3) {$\bullet$};
    
    \path[]
    (1) edge node {} (12)
    (12) edge node {} (23)
      (23) edge node {} (34)
    (34) edge node {} (25)
      (25) edge node {} (16)
    (16) edge node {} (7) ;
\end{tikzpicture}
\caption{$z^{0}z^{1}z^{2}$}
\end{center}
\end{figure}
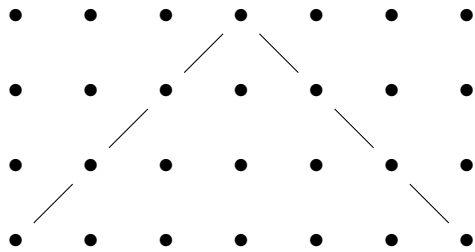

\end{document}